\documentclass{amsart}

\usepackage{lmodern}
\usepackage{microtype}

\usepackage{booktabs,array}
\newcolumntype{M}[1]{>$#1<$}

\usepackage{xparse}

\newcommand\gobbletwo[2]{}

\usepackage[b5paper,margin=0.8in]{geometry}

\usepackage[ngerman,english]{babel}

\defineshorthand{"_}{\penalty300\ }

\addto\extrasenglish{\useshorthands{"}}
\usepackage{amsthm}

\usepackage{amsmath,amssymb,mathtools,bm,textcomp,stmaryrd} % math
\usepackage{booktabs,array} % tables
\usepackage{graphicx,xcolor} % graphics, colors

\newif\ifTIKZ
\TIKZfalse
\ifTIKZ
  \usepackage{tikz} % TikZ
  \usetikzlibrary{arrows, arrows.meta, automata, calc, cd, hobby, positioning, shapes}
  \usetikzlibrary{external}
  \newcommand\TIKZinclude[1]{\tikzsetnextfilename{#1}\input{#1.tex}}
\else
  \makeatletter
  \def\Xincludegraphics#1{\begingroup\leavevmode\Xpgfexternalreaddpth{#1}\setbox1=\hbox{\includegraphics{#1}}%
    \ifdim\Xpgfretval=0pt \box1 \else\dimen0=\Xpgfretval\relax\hbox{\lower\dimen0 \box1 }\fi\endgroup}
  \newread\Xr@pgf@reada
  \def\Xpgfexternalreaddpth#1{\edef\Xpgfexternalreaddpth@restore{\noexpand\endlinechar=\the\endlinechar\space
    \noexpand\catcode`\noexpand\@=\the\catcode`\@\space}\def\Xpgfretval{0pt}\endlinechar=-1 \catcode`\@=11 %
    \openin\Xr@pgf@reada=#1.dpth \Xpgfincludeexternalgraphics@read@dpth\Xpgfexternalreaddpth@restore}
  \def\Xpgfincludeexternalgraphics@read@dpth{\ifeof\Xr@pgf@reada\closein\Xr@pgf@reada\else\read\Xr@pgf@reada
    to\Xpgfincludeexternalgraphics@auxline\ifx\Xpgfincludeexternalgraphics@auxline\empty\else
    \expandafter\Xpgfincludeexternalgraphics@read@dpth@line\Xpgfincludeexternalgraphics@auxline
    \Xpgfincludeexternalgraphics@read@dpth@line@EOI\fi\expandafter\Xpgfincludeexternalgraphics@read@dpth\fi}
  \long\def\Xpgfincludeexternalgraphics@read@dpth@line#1#2\Xpgfincludeexternalgraphics@read@dpth@line@EOI{%
    \ifcat\noexpand#1\relax\if@filesw{\toks0={#1#2}\immediate\write\@auxout{\noexpand\def\noexpand\Xdpthimport{%
    \the\toks0 }\noexpand\Xdpthimport }}\fi\else\def\Xpgfretval{#1#2}\fi}%
  \makeatother
  \newcommand\TIKZinclude[1]{\Xincludegraphics{#1}}
\fi

\DeclarePairedDelimiterX\set[2]\lbrace\rbrace{\,#1\suchthat#2\,} % set defined through a condition
\DeclarePairedDelimiter\class\llbracket\rrbracket % standard parentheses
\DeclarePairedDelimiter\qlb\lparen\rparen % standard parentheses
\DeclarePairedDelimiter\qfl\lfloor\rfloor % integer floor function
\DeclarePairedDelimiter\abs\lvert\rvert % absolute value
\DeclareDocumentCommand\Exp{ m s o m }{%
  \IfBooleanTF{#2}{% starred
    \qlb*{#4}_{\ExInd{#1}}
  }{% not starred
    \IfNoValueTF{#3}{% no []
      \qlb{#4}_{\ExInd{#1}}
    }{% data in []
      \qlb[#3]{#4}_{\ExInd{#1}}
    }
  }
}
\newcommand*\ExInd[1]{\mathsf{#1}} % Expansion
\newcommand*\EB{\Exp{B}}
\newcommand*\EBx{\Exp{\tilde B}}
\newcommand*\ES{\Exp{S}}
\newcommand*\Tra[1]{T_{\ExInd{#1}}} % Transformation
\newcommand*\TG{\Tra{G}}
\newcommand*\TB{\Tra{B}}
\newcommand*\TBx{\Tra{\tilde B}}
\newcommand*\TS{\Tra{S}}
\newcommand*\Dig[1]{D_{\ExInd{#1}}} % Digit function
\newcommand*\DB{\Dig{B}}
\newcommand*\DBx{\Dig{\tilde B}}
\newcommand*\DS{\Dig{S}}
\newcommand*\Xom[1]{X_{\ExInd{#1}}} % X-domain
\newcommand*\XB{\Xom{B}}
\newcommand*\XBx{\Xom{\tilde B}}
\newcommand*\XS{\Xom{S}}
\newcommand*\mue[1]{\mu_{\ExInd{#1}}} % invariant measure
\newcommand*\muB{\mue{B}}
\newcommand*\muS{\mue{S}}
\newcommand*\Rau[1]{\RR_{\ExInd{#1}}} % invariant measure
\newcommand*\RB{\Rau{B}}
\newcommand*\RS{\Rau{S}}
\newcommand*\Cond[1]{(\ExInd{#1})} % condition in proof of Theorem {thm:tiling}
\newcommand*\CB{\Cond{B}}
\newcommand*\CN{\Cond{N}}
\newcommand*\CS{\Cond{S}}
\DeclareMathOperator\supp{supp} % support
\newcommand*\N{\mathbb{N}} % natural numbers
\newcommand*\R{\mathbb{R}} % real numbers
\newcommand*\C{\mathbb{C}} % complex numbers
\newcommand*\Q{\mathbb{Q}} % rational numbers
\newcommand*\Z{\mathbb{Z}} % integers
\newcommand*\RR{\mathcal{R}} % beta-tile
\newcommand*\PP{\mathcal{P}} % purely periodic points
\newcommand*\LL{\mathcal{L}} % layer
\newcommand*\TT{\mathcal{T}} % tiling
\newcommand*\dd{\mathrm{d}} % differential
\def\Xmathscalebox#1#2#3{\scalebox{#1}{$#2#3$}}
\newcommand*\mathscalebox[2]{\mathpalette{\Xmathscalebox{#1}}{#2}}
\let\Dot\relax
\newcommand*\Dot{{{}_{\mathscalebox{0.8}{\mkern-1mu\bullet\mkern-1mu}}}} % fractional point
\newcommand*\defined[1]{\emph{#1}} % for defined terms
\newcommand*\suchthat{:} % alias (colon for set definitions)
\newcommand*\eqdef{\coloneqq} % alias (definition equals sign)
\newcommand*\xto{\xrightarrow[\quad]} % alias (limit arrow with overset)
\newcommand*\minusdigit[1]{\mkern1mu\overline{\mkern-1mu{#1}\mkern-1mu}\mkern1mu}
\newcommand*\1{{\minusdigit{1}}{}}

\newcommand*\conj[1]{#1^\star}

\makeatletter
\newcommand{\cdotfill}{%
  \mathinner{\!}
  \cleaders\hbox{$\m@th\cdotp{}$}\hfill
}
\makeatother

\makeatletter
\usepackage{etoolbox}
\let\xproof\proof

\def\yyproof[#1]{\xproof[Proof #1]}
\def\yproof{\@ifnextchar[\yyproof\xproof}
\let\proof\yproof

\makeatother

\numberwithin{equation}{section}

\makeatletter
\def\th@plain{\slshape}
\makeatother
\theoremstyle{plain}
\newtheorem{lemma}{Lemma}%[section]
\newtheorem{theorem}{Theorem}

\theoremstyle{definition}

\newtheorem{example}{Example}
\newtheorem{problem'}{Problem}
\theoremstyle{remark}

\usepackage{hyperref,url} % should be loaded last
\hypersetup{
}

\author{Tom\'a\v s Hejda}

\title{Multiple tilings associated to~$d$-Bonacci~beta-expansions}
\address{
	LIAFA, CNRS UMR 7089, Universit\'e Paris Diderot -- Paris 7, Case 7014, 75205 Paris, France
\newline\hspace*{\parindent}%
	Dept.\@ Math.\@ FCE, University of Chemistry and Technology Prague, Studentská 2031/6, 16000 Prague, Czechia
\newline\hspace*{\parindent}%
	Dept.\@ Algebra.\@ FMF, Charles University, Sokolovsk\'a~49/83, 18600 Prague, Czechia
}
\email{tohecz@gmail.com}

\begin{document}

\begin{abstract}
Let $\beta\in(1,2)$ be a Pisot unit and consider the symmetric $\beta$-expansions.
We give a necessary and sufficient condition for the associated Rauzy fractals to form a tiling
 of the contractive hyperplane.
For $\beta$ a $d$-Bonacci number, i.e., Pisot root of $x^d-x^{d-1}-\dots-x-1$ we show that the Rauzy fractals form a multiple tiling with covering degree $d-1$.
\end{abstract}

\maketitle

\section{Introduction}

Tilings arising from $\beta$-expansions were first studied in the 1980s by A.~Rauzy"_\cite{rauzy_1982} and W.~Thurston"_\cite{thurston_1989}.
They consider the greedy $\beta$-expansions that are associated to the transformation
 $\TG \colon x\mapsto \beta x-\qfl{\beta x}$.
S.~Akiyama"_\cite{akiyama_2002} showed that the collection of $\beta$-tiles forms a tiling
 if and only if $\beta$ satisfies the so-called \defined{weak finiteness property}~(W).
M.~Barge"_\cite{barge_2016b,barge_2016a} proved that all Pisot numbers satisfy property~(W);
 he actually proves that the $\beta$-substitution associated to the greedy transformation has pure discrete spectrum.

If we drop the ``greedy'' hypothesis, things are getting more interesting.
C.~Kalle and W.~Steiner"_\cite{KS} showed that the symmetric $\beta$-expansions
 for two particular cubic Pisot numbers~$\beta$ induce
 a double tiling --- i.e., a multiple tiling such that almost every point of the tiled space lies in exactly two tiles.
More generally, they proved that every ``well-behaving''
 $\beta$-transformation with a Pisot unit $\beta$ induces a multiple tiling.
The method of Barge cannot be straightforwardly extended to the symmetric $\beta$-expansions, because
 for these, there is no direct link to Pisot substitutions.
The provided examples of multiple tilings therefore do not disprove the general Pisot substitution
 conjecture that all Pisot irreducible substitutions have pure discrete spectrum.
Actually, multiple tilings are also considered for combinatorial substitutions
 as done e.g.\@ by S.~Ito and H.~Rao"_\cite{ito_rao_2006}.
We refer to a~\emph{m\'emoire} by A.~Siegel and J.~Thuswaldner"_\cite{siegel_thuswaldner_2009} for a thorough survey
 on substitution tilings and the Pisot conjecture.

In this paper we concentrate on the symmetric $\beta$-expansions associated
 to the transformation $\TS\colon x\mapsto \beta x-\qfl{\beta x+\frac12}$.
This transformation was studied before e.g.\@ by S.~Akiyama and K.~Scheicher
 in the context of shift radix systems"_\cite{akiyama_scheicher_2007}.
We consider $\beta\in(1,2)$ and we define $\TS$ on two intervals
 $[-\frac12,\tfrac\beta2-1)\cup[1-\tfrac\beta2,\frac12)$.
We show the following theorem about the multiple tiling:

\begin{theorem}\label{thm:mt}
Let $d\in\N$, $d\geq2$, and let $\beta\in(1,2)$ be the $d$-Bonacci number, i.e.,
 the Pisot number satisfying $\beta^d = \beta^{d-1} + \dots + \beta + 1$.
Then the symmetric $\beta$-expansions induce a multiple tiling of $\R^{d-1}$ with covering degree equal to $d-1$.
\end{theorem}

It was shown before by H.~Rao, Z.-Y.~Wen and Y.-M.~Yang"_\cite[Theorem~1.6]{rao_wen_yang_2014}
 that the covering degree is a multiple of $d-1$.
We also note that for any particular $\beta$ and any particular transformation, the degree of the multiple tiling
 can be computed from the intersection (or boundary) graph, eventually multi-graph,
 as defined for instance by A.~Siegel and J.~Thuswaldner"_\cite{siegel_thuswaldner_2009};
 however, such an algorithmic approach is not usable for an infinite number of cases.

We also characterize the tiles that form the distinct layers of the multiple tiling:

\begin{theorem}\label{thm:L}
Let $d\in\N$, $d\geq3$, and let $\beta\in(1,2)$ be the $d$-Bonacci number.
Let $h\in\{1,2,\dots,d-1\}$.
Then the collection of tiles $\set{\RR(x)}{x\in\LL_h}$, where
\begin{equation}\label{eq:L}
	\LL_h \eqdef \Bigl(\class{h}\cap\bigl[1-\tfrac\beta2,\tfrac12\bigr)\Bigr)
	\cup \Bigl(\class{h-1}\cap\bigl[-\tfrac12,\tfrac\beta2-1\bigr)\Bigr)
,\end{equation}
 forms a tiling of $\R^{d-1}$, that is, it is a layer of the multiple tiling guaranteed by Theorem~\ref{thm:mt}.
Here we denote $\class{j} \eqdef j+(\beta-1)\Z[\beta]$.
\end{theorem}

The two results rely substantially on the knowledge of the purely periodic integer points of~$\TS$:

\begin{theorem}\label{thm:P}
Let $d\in\N$, $d\geq2$, and let $\beta\in(1,2)$ be the $d$-Bonacci number.
Let $\PP$ denote the set of non-zero $x\in\Z[\beta]$ such that $\TS^p x=x$ for some $p\geq1$.
Then
\[
	\PP \cup \{0\} = \set[\big]{\pm\,\Dot 0p_2p_3\dotsm p_d}{p_i\in\{0,1\}} = \set[\Big]{\pm\sum\limits_{i=2}^d p_i\beta^{-i}}{p_i\in\{0,1\}}
.\]
\end{theorem}
\noindent
(We exclude $0$ from $\PP$ as it does not lie in
 the support of the invariant measure of~$\TS$.)

\medskip

Last but not least, for general Pisot units $\beta\in(1,2)$
 we give a necessary and sufficient condition on the tiling property for the symmetric $\beta$-expansions:

\begin{theorem}\label{thm:tiling}
Let $\beta\in(1,2)$ be a Pisot unit of degree $d\geq2$.
Then the symmetric $\beta$-expansions induce a tiling of $\R^{d-1}$
 if and only if the following two conditions are satisfied:
\begin{enumerate}
 \item $\beta-1$ is an algebraic unit (i.e., $N(\beta-1)=\pm1$);
 \item the balanced $\beta$-expansions induce a tiling of $\R^{d-1}$
  (the balanced expansions are defined in \S\,\ref{sect:beta}).
\end{enumerate}
\end{theorem}

The paper is organized as follows.
In the following section we define all the necessary notions.
The theorems are proved in Section~\ref{sect:proofs}.
We conclude by a pair of related open questions in Section~\ref{sect:problems}.

\section{Preliminaries}\label{sect:pre}

\subsection{Pisot numbers}

An algebraic integer $\beta>1$ is a \defined{Pisot number} iff all its Galois conjugates, i.e., the other roots of its minimal polynomial,
 lie inside of the unit complex circle.
As usual, $\Z[\beta]$ denotes the ring of integer combinations of powers of $\beta$,
 and $\Q(\beta)$ denotes the field generated by the rational numbers and by~$\beta$.

Suppose that $\beta$ is of degree $d$ and has
 $2e<d$ complex Galois conjugates $\beta_{(1)},\dots, \beta_{(e)}, \allowbreak \conj{\beta_{(1)}},\dots, \conj{\beta_{(e)}}$
 (where $\conj{z}$ denotes the complex conjugate of $z$)
 and $d-2e-1$ real ones $\beta_{(e+1)},\dots,\beta_{(d-e-1)}$.
Denote $\sigma_{(j)}\colon \Q(\beta)\to\Q(\beta_{(j)})$ the corresponding Galois isomorphisms.
Then we put
\[
	\Phi\colon \Q(\beta)\to \!\!\prod_{j=1}^{d-e-1}\!\! \Q(\beta_{(j)})
	,\quad
	x\mapsto \bigl(\sigma_{(1)}(x),\dots,\sigma_{(d-e-1)}(x)\bigr)
.\]
Since $\prod_{j=1}^{d-e-1} \Q(\beta_{(j)})\subset \C^e\times \R^{d-2e-1}\simeq \R^{d-1}$,
 we consider that $\Phi\colon \Q(\beta)\to\R^{d-1}$.
We have the closure properties $\overline{\Phi(\Z[\beta])} = \overline{\Phi(\Q(\beta))} = \R^{d-1}$;
 this follows from the Strong Approximation Theorem~\cite[Ch.~3, \S\,1, Exercise~1]{neukirch_1999}
 and from the fact that $\Z[\beta]$ has finite index in the ring of integers of $\Q(\beta)$.

\medskip

In this paper, we focus on $d$-Bonacci numbers.
For $d\geq2$ a \defined{$d$-Bonacci number} is the Pisot root of the polynomial $p_d(x)=x^d-x^{d-1}-\dots-x-1$.
A.~Brauer"_\cite{brauer_1951} showed that this polynomial is irreducible and has a Pisot root.
This root satisfies $\beta\in(1,2)$ because $p_d(1)=-(d-1)$ and $p_d(2)=1$ have the opposite signs.

\medskip

We say that two numbers $x,y\in\Z[\beta]$ are \defined{congruent modulo $\beta-1$} iff $y-x\in(\beta-1)\Z[\beta]$.
By $\class{h}$, for $h\in\Z[\beta]$, we denote the congruence class modulo $\beta-1$ that contains $h$,
 i.e., $\class{h} \eqdef h+(\beta-1)\Z[\beta]$.
If $\beta$ is a $d$-Bonacci number, then the norm of $\beta-1$ is $N(\beta-1)=\pm(d-1)$.
Therefore there are exactly $d-1$ distinct classes modulo $\beta-1$
 and we can take numbers $h\in\{1,2,\dots,d-1\}$ as their representatives, i.e.,
\[
	\Z[\beta] = \bigcup_{h=1}^{d-1} \class{h}
	= \bigcup_{h=1}^{d-1} h + (\beta-1)\Z[\beta]
.\]

\subsection{\texorpdfstring{$\beta$}{Beta}-expansions}\label{sect:beta}

We fix $\beta\in(1,2)$.
Let $X\subset\R$ be a union of intervals and $D\colon X\mapsto \Z$ be a piecewise constant function (\defined{digit function}) such that
 $\beta x-D(x)\in X$ for all $x\in X$.
Then the map $T\colon X\to X,\,x\mapsto \beta x-D(x)$ is a \defined{$\beta$-transformation}.
The \defined{$\beta$-expansion} of $x\in X$ is then the (right-infinite) sequence $x_1x_2x_3\dotsm \in (D(X))^\omega$,
 where $x_i = D(T^{i-1})x$.
We say that $x_1x_2x_3\dotsm\in\Z^\omega$ is \defined{$T$-admissible} iff it is the expansion of some $x\in X$.

We define two particular $\beta$-transformations:

\begin{enumerate}

\item Let $\XS \eqdef [-\frac12, \frac\beta2-1)\cup[1-\frac\beta2,\frac12)$ and $\DS(x) \eqdef \qfl{\beta x-\frac12}\in\{\1,0,1\}$ (we denote $\overline{a}\eqdef -a$ for convenience).
 This defines the \defined{symmetric $\beta$-expansions}.
 We denote $\TS$ the transformation and $\ES{x}\in\{\1,0,1\}^\omega$ the expansion of $x\in\XS$.

\item Let $\XB \eqdef [\frac{2-\beta}{2\beta-2},\frac{\beta}{2\beta-2})$ and $\DB(x) \eqdef 1$ iff $x\geq\frac{1}{2\beta-2}$ and $\DB(x) \eqdef 0$ otherwise.
 This defines the \defined{balanced $\beta$-expansions}.
 We denote $\TB$ and $\EB{x}\in\{0,1\}^\omega$ accordingly.

\end{enumerate}
Both $\TS$ and $\TB$ are plotted in Figure~\ref{fig:TSB} for the Tribonacci number.

Besides expansions, we consider arbitrary representations.
Any bounded sequence of integers $x_{-N}\dotsm \allowbreak x_{-1}x_0\Dot x_1x_2\dotsm$%, with $x_i\in\Z$ bounded,
 is a \defined{representation} of $x = \sum_{i\geq-N} x_{i}\beta^{-i}\in\R$.

A~\defined{factor} of a sequence $x_1x_2x_3\dotsm$ is any finite word $x_k x_{k+1}\dotsm x_{l-1}$ with $l\geq k\geq1$.
A~\defined{tail} of a sequence $x_1x_2x_3\dotsm$ is any of the infinite words $x_kx_{k+1}x_{k+2}\dotsm$ for $k\geq1$.
A~sequence $x_1x_2\dotsm$ is \defined{periodic} iff $(\exists k,p\in\N,\, p\geq1)(\forall i>k)(x_{i+p}=x_i)$.
It is \defined{purely periodic} iff $k=0$.

\subsection{Rauzy fractals}\label{sect:Rauzy}

We consider the symmetric $\beta$-transformations for Pisot units $\beta$.
The symmetric $\beta$-transformation $\TS$ possesses a unique absolutely continuous invariant measure
 (w.r.t.\@ the Lebesgue measure).
This follows from the work of T.-Y.~Li and J.~Yorke"_\cite[Theorem~1]{li_yorke_1978},
 because while $\TS$ has more than one discontinuity point when $\beta>2$, we have that $\TS x=\frac12$
 for all these points, therefore each $L_i$ in the theorem statement must contain $\frac12$ in its iterior;
 this means that there is only $L_1$.
For any $x\in\Z[\beta]\cap\XS$, we define the \defined{$\beta$-tile} (or \defined{Rauzy fractal})
 as the Hausdorff limit
\[
	\RR(x) \eqdef \lim_{n\to\infty} \Phi\bigl(\beta^n \TS^{-n}(x)\bigr) \subset \R^{d-1}
.\]
Note that $\TS^{-n}(-x) = -\TS^{-n}(x)$ for all $x\in\Z[\beta]\cap\XS$ and all $n$
 (this holds as the boundary points of the intervals are not in $\Z[\beta]$, nor are their images under $\TS$),
 therefore $\RR(-x)=-\RR(x)$.

The Rauzy fractals induce a multiple tiling, as will follow from the work of Kalle and Steiner"_\cite[Theorem~4.10]{KS}.
We recall that the family of tiles $\TT\eqdef\{\RR(x)\}_{x\in\Z[\beta]\cap\XS}$
 is a \defined{multiple tiling} iff the following is satisfied:
\begin{enumerate}
\item The tiles $\RR(x)$ take only finitely many shapes (i.e., are only finitely many modulo translations in $\R^{d-1}$).
\item The family $\TT$ is locally finite, i.e., for every bounded set $U\subset\R^{d-1}$, only finitely many tiles from $\TT$ intersect $U$.
\item The family $\TT$ covers $\R^{d-1}$, i.e., for every $y\in\R^{d-1}$ there exists $\RR(x)\in\TT$ such that $y\in\RR(x)$.
\item Every tile $\RR(x)$ is a closure of its interior.
\item There exists an integer $m\geq1$ such that almost every point in $\R^{d-1}$ lies in exactly $m$ tiles from $\TT$;
 this $m$ is called the \defined{covering degree} of~$\TT$.
\end{enumerate}
If $m=1$, we say that $\TT$ is a \defined{tiling}.
Every multiple tiling with covering degree $m\geq2$ is a union of $m$ tilings;
 we call these tilings \defined{layers} of the multiple tiling.

\section{Proofs}\label{sect:proofs}

First, we establish a strong relation between the symmetric and the balanced expansions in Lemma~\ref{lem:B-S};
 this works for all $\beta\in(1,2)$.

Then, we suppose that $d\geq3$ is an integer and $\beta\in(1,2)$ is the $d$-Bonacci number.
In Lemma~\ref{lem:inv-m} we show that the support of the invariant measure of $\TS$ is the whole $\XS$;
 from this, we conclude that $\{\RR(x)\}_{x\in\Z[\beta]\cap\XS}$ is a multiple tiling"_\cite[Theorem 4.10]{KS}.
Then we investigate arithmetic properties of the balanced expansions in Lemmas~\ref{lem:admis}, \ref{lem:1} and~\ref{lem:h}.
We use these properties to determine the degree of the multiple tiling,
 which is done in Lemmas~\ref{lem:leq}, \ref{lem:KS} and~\ref{lem:geq}.
The proof of Theorem~\ref{thm:P} is given after Lemma~\ref{lem:h},
 the proofs of Theorems~\ref{thm:mt} and~\ref{thm:L} are after Lemma~\ref{lem:geq}.

We close this section by the proof of Theorem~\ref{thm:tiling}.

\begin{lemma}\label{lem:B-S}
%Let $x\in\Z[\beta]\cap\XS$.
Let $\beta\in(1,2)$.
Define a bijection
\[
	\psi\colon\XS\to\XB,
	\quad
	x\mapsto \begin{cases}
		\frac{1}{\beta-1}x & \text{if $x\in[1-\frac\beta2,\frac12)$}
	,\\
		\frac{1}{\beta-1}(x+1) & \text{if $x\in[-\frac12,\frac\beta2-1)$}
	.\end{cases}
\]
Suppose that $\EB{\psi x}=t_1t_2t_3\dotsm$.
Then $\ES{x}=(t_2{-}t_1)(t_3{-}t_2)(t_4{-}t_3)\dotsm$.
Moreover, $\ES{x}$ is purely periodic if and only if $\EB{\psi x}$ is,
 and the length of the periods is the same.
\end{lemma}

\begin{proof}
The transformations $\TS$ and $\TB$ are conjugated via $\psi$, i.e., the following diagram commutes:
\[
	\shorthandoff{"}
	\vcenter{\hbox{\TIKZinclude{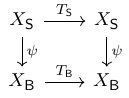}}}
	\shorthandon{"}
\]
 (see Figure~\ref{fig:TSB}).
We partition $\XS$ into $I_\1\eqdef[-\frac12,-\frac{1}{2\beta})$, $I_{0-}\eqdef[-\frac{1}{2\beta},\frac\beta2-1)$,
 $I_{0+}\eqdef[1-\frac\beta2,\frac{1}{2\beta})$ and $I_1\eqdef[\frac{1}{2\beta},\frac12)$.
We also denote $I_0=I_{0-}\cup I_{0+}$.
Then we have $\DS(I_d)=d$ for $d\in\{\1,0,1\}$.

We similarly partition $\XB$, as depicted in Figure~\ref{fig:TSB} right.
Then
\[
	\psi I_\1 = J_{10}
,\quad
	\psi I_{0-} = J_{11}
,\quad
	\psi I_{0+} = J_{00}
,\quad
	\psi I_1 = J_{01}
,\]
 therefore we see that if $\psi x\in J_{ab}$ then $x\in I_{(b-a)}$.
Finally, we see that for $a,b\in\{0,1\}$ we have that $\DB(J_{ab})=a$ and $\TB(J_{ab})\subseteq J_{b0}\cup J_{b1}$ hence $\DB(\TB(J_{ab}))=b$.
This means that if $\TB^i(\psi x)\in J_{ab}$ then $t_{i+1}t_{i+2}=ab$ and also $\TS^i(x)=\psi^{-1}\TB^i(\psi x)\in I_{(b-a)}$ hence $x_{i+1}=b-a=t_{i+2}-t_{i+1}$.

The periodicity is preserved because $
	\TS^p x = x
	\Longleftrightarrow
	\TB^p \psi x = \psi x
$.
\end{proof}

\begin{lemma}\label{lem:inv-m}
Let $\beta$ be a $d$-Bonacci number.
The support of the invariant measure of $\TS$ is the whole domain
 $\XS=[-\frac12, \frac\beta2-1)\cup[1-\frac\beta2,\frac12)$.
\end{lemma}

\begin{figure}
\centering
\TIKZinclude{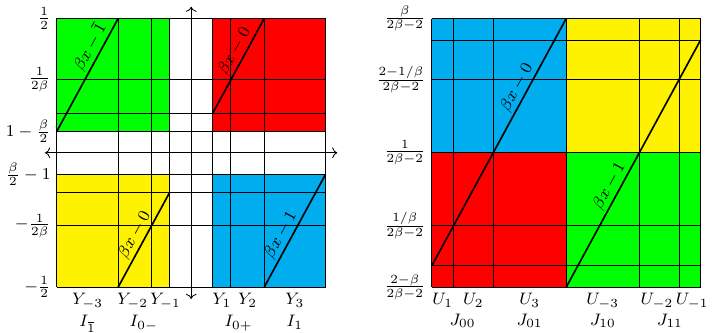}%
\caption{Transformations $\TS$ (left) and $\TB$ (right) for $d=3$.}
\label{fig:TSB}
\end{figure}

\begin{proof}
Denote $l \eqdef -\frac12$.
Put $Y_d \eqdef [\TS^d l,l+1)$ and $Y_k \eqdef [\TS^k l, \TS^{k+1} l)$ for $1\leq k\leq d-1$.
Similarly, put $Y_{-d} \eqdef [l,-\TS^d l)$ and $Y_{-k} \eqdef [-\TS^{k+1} l, -\TS^k l)$ for $1\leq k\leq d-1$,
 see Figure~\ref{fig:TSB}.

Define a measure $\mu$ by
\[
	\dd\mu(x) = f(x)\,\dd x \eqdef \Bigl(\frac{1}{\beta} + \frac{1}{\beta^2} + \dotsm + \frac{1}{\beta^k}\Bigr) \dd x
	\quad\text{for $x\in Y_{\pm k}$, $1\leq k\leq d$}
.\]
Then we verify that for any $x\in\XS$, we have
\[
	\mu\bigl([x, x+\dd x)\bigr)
	= f(x) \, \dd x
	= \frac1\beta \, \dd x
		\sum_{\substack{y\in\XS\\ \TS y=x}}
		f(y)
	= \mu\bigl(\TS^{-1}[x,x+\dd x)\bigr)
,\]
because
\begin{equation}\label{eq:TSXk}
	\TS Y_{\pm k} = \begin{cases}
		Y_{\mp 1} \cup Y_{\mp 2} \cup\dots\cup Y_{\mp d}
		& \text{if $k=d$}
	,\\
		Y_{\pm(k+1)}
		& \text{otherwise}
	.\end{cases}
\end{equation}
Therefore $\mu$ is the invariant measure of $\TS$.
\end{proof}

%\begin{figure}[!t]
%\centering
%\TIKZinclude{tilebo-ASB}
%\caption{The automata accepting the $\TS$-admissible sequences (top) and the $\TB$-admissible ones (bottom).}
%\label{fig:ASB}
%\end{figure}

\begin{lemma}\label{lem:admis}
Let $\beta$ be a $d$-Bonacci number.
A sequence $x_1x_2x_3\dotsm$ is $\TB$-admissible if and only if it contains neither $0^{d+1}$ nor $1^{d+1}$ as a factor
 and it does not have $(1^d0)^\omega$ as a tail.
\end{lemma}

\begin{proof}
We will rely on the generalized Parry condition~\cite[Theorem 2.5]{KS}.
We have that $x_1x_2x_3\dotsm$ is $\TB$-admissible if and only if for all $i\geq1$ we have
\begin{align*}
	\EB{l} &\stackrel{(\textrm A_i)}\preceq x_ix_{i+1}x_{i+2}\dotsm \stackrel{(\textrm B_i)}\prec \EBx{l+\tfrac12}
	\quad\text{if } x_i=0
,\\
	\EB{l+\tfrac12} &\stackrel{(\textrm C_i)}\preceq x_ix_{i+1}x_{i+2}\dotsm \stackrel{(\textrm D_i)}\prec \EBx{l+1}
	\quad\text{if } x_i=1
,\end{align*}
 where $l=\frac{2-\beta}{2\beta-2}$ and $\EBx{x}$ is the expansion of $x$ w.r.t.\@ transformation $\TBx$ defined on $\XBx\eqdef(l,l+1]$
 with digit function $\DBx(y)=1$ if $y>l+\frac12=\frac{1}{2\beta-2}$ and $\DBx(y)=0$ otherwise.
Here we denote $(\prec)$ the lexicographic ordering on $\{0,1\}^\omega$.
We have that
\[
	\EB{l}=(0^d1)^\omega
,\quad
	\EBx{l+\tfrac12}=(01^d)^\omega
,\quad
	\EB{l+\tfrac12}=(10^d)^\omega
,\quad
	\EBx{l+1}=(1^d0)^\omega
.\]
Note that $(\textrm B_i)\Leftarrow (\textrm D_{i+1})$ and $(\textrm C_i)\Leftarrow(\textrm A_{i+1})$.

Direction ($\Leftarrow$).
Suppose $x_1x_2x_3\dotsm$ does not contain either of the two forbidden factors nor the forbidden tail.
We need to show that conditions $(\textrm A_i)$ and $(\textrm D_i)$ are satisfied.
Fix $i\geq1$.
From the absence of $0^{d+1}$ we know that either $x_ix_{i+1}\dotsm=(0^d1)^\omega$
 or it has a prefix $(0^d1)^j0^q1$ with $q\leq d-1$.
Either way, $(\textrm A_i)$ is satisfied.
Similarly, from the absence of $1^{d+1}$ and $(1^d0)^\omega$ we derive that $(\textrm D_i)$ is satisfied.
Therefore the sequence is $\TB$-admissible.

Direction ($\Rightarrow$).
We know that $(\mathrm A_i)$ and $(\mathrm D_i)$ are satisfied by the $\TB$-expansion of any $x\in\XB$.
Now, $(\mathrm A_i)$ forbids $0^{d+1}$ as a factor since any sequence starting with $0^{d+1}$
 is lexicographically smaller than $(0^d1)^\omega$.
Similarly, $(\mathrm D_i)$ forbids $1^{d+1}$.
The forbidenness of the tail $(1^d0)^\omega$ follows from the strict inequality in $(\mathrm D_i)$.
\end{proof}

\begin{lemma}\label{lem:1}
Let $\beta$ be a $d$-Bonacci number.
Suppose that the balanced expansion of $x\in\Q(\beta)\cap\XB$ has the form
\[
	\EB{x} = x_1 x_2 x_3\dotsm x_n (x_{n+1}\dotsm x_{n+d})^\omega
.\]
Then for any $z\in\Z[\beta]$ such that $x+z\in\XB$, the balanced expansion of $x+z$ has the form
\[
	\EB{x+z} = y_1 y_2 y_3\dotsm y_m (y_{m+1}\dotsm y_{m+d})^\omega
,\]
 where, moreover, $x_{n+1}+\dots+x_{n+d} = y_{m+1}+\dots+y_{m+d}$.
\end{lemma}

\begin{proof}
Clearly it is enough to consider the simplest case $z=\pm\beta^{-k}$ for some $k\geq2$,
 since any $z\in\Z[\beta]$ is a finite sum of powers of $\beta$.
Then $x+z = \Dot \tilde x_1 \tilde x_2 \tilde x_3\dotsm$, where $\tilde x_i = x_i$ for $i\neq k$, and $\tilde x_k = x_k\pm1$.
Let $y_1y_2y_3\dotsm$ be the balanced expansion of $x+z$.

Denote
\[
	s_i
		\eqdef
		\underbrace{\Dot y_{i+1}y_{i+2}y_{i+3}\dotsm}_{\in[l,l+1)}
		- \underbrace{\Dot \tilde x_{i+1}\tilde x_{i+2}\tilde x_{i+3}\dotsm}_{\in[l-\frac1\beta,l+1+\frac1\beta)}
,\]
 where we put $l\eqdef \frac{2-\beta}{2\beta-2}$.
Then $s_0=0$ and $s_i\in(-1-\frac1\beta,1+\frac1\beta)$, and we have that $s_{i+1} = \beta s_i + (\tilde x_{i+1}-y_{i+1})$;
 we will denote this relation by a labelled arrow $s_i\xto{\tilde x_{i+1}-y_{i+1}}s_{i+1}$.

Consider $i\leq k-2$.
Then the only possible values of $s_i$ and possible arrows are:
\begin{gather}
	0 \xto{0} 0
,\qquad
 	0 \xto{\pm1} \pm\,\Dot 1^d
,\qquad
 	\pm\,\Dot 1 \xto{0} \pm\,\Dot 1^d
%,\qquad
,\nonumber\\
 	\pm\,\Dot 1^q\xto{\mp1} \pm\,\Dot 1^{q-1}
\quad\text{($1\leq q\leq d$)}
.\label{eq:arrows-k2}\end{gather}
Consider $i=k-1$.
On one hand, we know that $s_i\in(-1,1)$ because $\Dot \tilde x_{i+1}\tilde x_{i+2}\dotsm\in[l,l+1)$,
 which makes $\pm1=\pm\,\Dot1^d$ unreachable.
On the other hand, we have some additional arrows labelled $\pm2$, namely
\[
	\pm\,\Dot 1^q\xto{\mp2} \mp\,\Dot 0^{q-1}1^{d-q+1}
\quad\text{($1\leq q\leq d$)}
.\]
For $i\geq k$, we have $s_i\in(-1,1)$.
We have to check where the possible values $s_i=\pm\,\Dot0^q1^r$ lead us; we get:
\[\begin{aligned}
	\pm\,\Dot 0^q1^r &\xto{0} \pm\,\Dot 0^{q-1}1^r
&&\text{($1\leq q,r\leq d-1$ and $q+r\leq d$)}
,\\
	\pm\,\Dot 0^q1^r &\xto{\mp1} \mp\,\Dot 1^{q-1}0^r1^{d-q-r+1}
&&\text{($1\leq q,r\leq d-1$ and $q+r\leq d$)}
,\\
	\pm\,\Dot 1^q0^r1^t &\xto{\mp1} \pm\,\Dot 1^{q-1}0^r1^t
&&\text{($1\leq q,r,t\leq d-1$ and $q+r+t\leq d$)}
.\end{aligned}\]
(We easily verify that no other arrows are reachable by showing that for any other pair of $(s_i, \tilde x_{i+1}-y_{i+1})$,
 where $s_i$ is already included in the lists above, we get that $\beta s_i + (\tilde x_{i+1}-y_{i+1})$ does not lie in the required intervals.)

\begin{figure}
\centering
\TIKZinclude{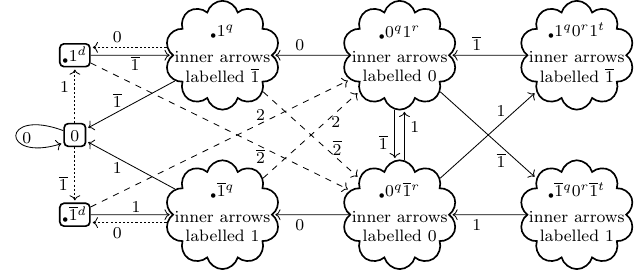}

\vspace{0.2in}

\TIKZinclude{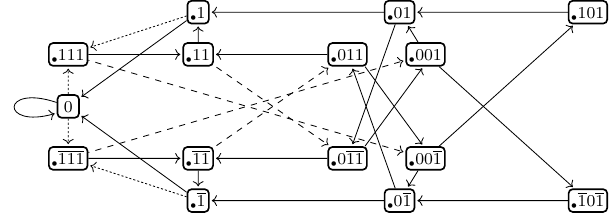}

\vspace{0.2in}

\TIKZinclude{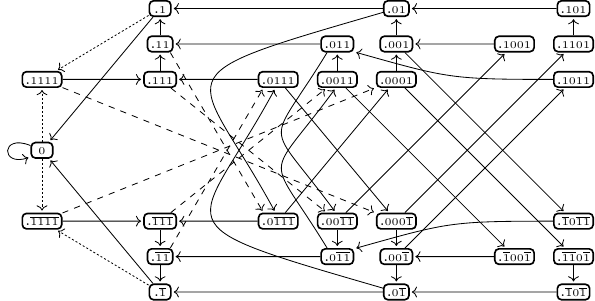}
\caption{The ``automaton'' built in the proof of Lemma~\ref{lem:1}.
 The dotted arrows are available only for $i\leq k-2$, the dashed arrows are available only for $i=k-1$.
 We depict the general case (top), the case $d=3$ (middle) and the case $d=4$ (bottom).}
\label{fig:plus}
\end{figure}

Let us now investigate the properties of the graph of all possible arrows.
A~schematic view of the arrows is given in Figure~\ref{fig:plus}.
Note that the inner arrows that live inside each ``cloud'' do not form cycles, therefore
 sooner or later, any walk through the graph exits a cloud.
Based on the solid arrows in the graph, we conclude that either $s_i=0$ eventually,
 or $s_i=\Dot 0^q1^r$ infinitely many times.
If $s_i=0$ eventually, we get that $x_{i+1}x_{i+2}x_{i+3}\dotsm = y_{i+1}y_{i+2}y_{i+3}\dotsm$,
 which finishes the proof.
Otherwise, fix $i\geq \max\{k,n\}$ such that $s_i=\Dot 0^q1^r$.
Then $\tilde x_{i+1}\tilde x_{i+2}\tilde x_{i+3}\dotsm$ is purely periodic,
 i.e., $\tilde x_{i+1}\tilde x_{i+2}\tilde x_{i+3}\dotsm = (p_1p_2\dots p_d)^\omega$ for some $p_j\in\{0,1\}$.

There are two cases.
First, suppose $p_{q+1}p_{q+2}\dotsm p_{q+r} = 0^r$. Then
\[
	y_{i+1}y_{i+2}y_{i+3}\dotsm = p_1p_2\dots p_q 1^r p_{q+r+1}\dotsm p_d (p_1\dotsm p_d)^\omega
.\]
Second, suppose $p_{q+1}p_{q+2}\dotsm p_{q+r} \neq 0^r$. Then we can find unique $t,u$
 with $1\leq t\leq q$ and $1\leq u\leq r$ such that
\[
	p_t p_{t+1} \dotsm p_q = 01^{q-t}
\quad\text{and}\quad
	p_{q+u} p_{q+u+1} \dotsm p_{q+r} = 1 0^{r-u}
\]
 (if we had $p_1p_2\dotsm p_q = 1^q$, it would be a contradiction with
 $\Dot y_{i+1}y_{i+2}\dotsm = s_i + \Dot \tilde x_{i+1}\tilde x_{i+2}\dotsm <l+1$).
Then the new pre-period and period are
\begin{equation}\label{eq:yi}
	\begin{array}{@{}r@{}l@{}l@{}l@{}}
		y_{i+1} y_{i+2} \dotsm y_{i+d} = {}&
		p_1\dotsm p_{t-1}1 0^{q-t} p_{q+1}\dotsm p_{q+u-1} & 0 1^{r-u} p_{q+r+1}\dotsm p_d &
	,\\
		y_{i+d+1} y_{i+d+2}\dotsm = \bigl( &
		p_1\dotsm p_{t-1} 1 p_{t+1} \,\cdotfill    p_{q+u-1} & 0 p_{q+u+1} \cdotfill     p_d & \bigr)^\omega
	,\end{array}
\end{equation}
 because this value of the sequence $y_{i+1}y_{i+2}\dotsm$ is $\TB$-admissible and satisfies that
\begin{multline*}
	\Dot y_{i+1} y_{i+2} \dotsm - \Dot \tilde x_{i+1} \tilde x_{i+2} \dotsm 
\\
	= \Dot 0^{t-1}1\1^{q-t}0^{u-1}\11^{r-u}0^{d-q-r}\bigl(0^{t-1}10^{q+u-t-1}\10^{d-q-u}\bigr)^\omega
\\
	= \Dot 0^t\1^{q-t}0^r1^{r-u} + \Dot\bigl(0^{t-1}10^{q+u-t-1}\10^{d-q-u}\bigr)^\omega
\\
	= \Dot 0^t\1^{q-t}0^r1^{r-u} + \Dot0^t1^{q+u-t} = \Dot 0^q1^r = s_i
.\end{multline*}
In either case, the sum of the elements of the period is preserved.
\end{proof}

\begin{example}
We apply the lemma to an example $d=3$, $\EB{x}=0111011(010)^\omega$ and $z=\beta^{-7}$.
Then $\tilde x_1 \tilde x_2 \dotsm = 0111012(010)^\omega$ and
 $y_1 y_2 \dotsm = 1000100101(100)^\omega$.
The computation is as follows:
\[
\begin{tabular}{Ml*{7}{Mc@{\ }}Mc*{5}{Mc@{\ }}Mc@{\quad}>{$\dotsm$}l}
i & 0 & 1 & 2 & 3 & 4 & 5 & 6 & 7 & 8 & 9 & 10 & 11 & 12 & 13 &\\
	\cmidrule(r){1-9}
	\cmidrule(l){10-16}
\tilde x_i && 0 & 1 & 1 & 1 & 0 & 1 & 2 & 0 & 1 & 0 & 0 & 1 & 0 &\\
y_i && 1 & 0 & 0 & 0 & 1 & 0 & 0 & 1 & 0 & 1 & 1 & 0 & 0 &\\
s_i & \Dot0 & \Dot\1\1\1 & \Dot\1\1 & \Dot\1 & \Dot0 & \Dot\1\1\1 & \Dot\1\1 &
	\Dot011 & \Dot00\1 & \Dot101 &
	\Dot01 & \Dot0\1\1 & \Dot001 & \Dot01 &\\
\end{tabular}%
\]
 (this computation follows the arrows in Figure~\ref{fig:plus} middle).
For $i=7$, we have that $s_7=\Dot011$ and $x_8 x_9\dotsm = (100)^\omega$ is purely periodic.
Therefore we have $q=1$ and $r=2$ and $p_1p_2p_3=010$.
We have $p_{q+1}\dotsm p_{q+r}=10\neq0^r$; we get $t=1$ and $u=1$.
From \eqref{eq:yi} we confirm that $y_8y_9\dotsm=101(100)^\omega$.
\end{example}

\begin{lemma}\label{lem:h}
Let $\beta$ be a $d$-Bonacci number.
Let $h\in\{1,2,\dots,d-1\}$.
Then the set $\class{h}\cap\XS$ contains exactly such $x\in\Q(\beta)\cap\XS$
 that the balanced expansion of $\frac{\abs x}{\beta-1}$ has the form
\begin{gather}\label{eq:x-per}
	\EB[\Big]{\frac{\abs x}{\beta-1}} = x_1x_2\dotsm x_{n}(x_{n+1}x_{n+2}\dotsm x_{n+d})^\omega
\\\notag
	\text{with}\quad
	\!
	x_{n+1}+x_{n+2}+\dots+x_{n+d} = \begin{cases}
		h &\text{if $x>0$}
	,\\
		d-1 &\text{if $x<0$ and $h=d-1$}
	,\\
		d-1-h\!{} &\text{if $x<0$ and $1\leq h\leq d-2$}	
	. \end{cases}
\end{gather}
\end{lemma}

\begin{proof}
We start by proving that whatever $x\in\class{h}\cap\XS$ we take, it satisfies \eqref{eq:x-per}.
As $\Dot1^j\in\class{j}$ for all $j\in\N$, there exists $y\in\Z[\beta]$ such that
\[
	x = \begin{cases}
		(\beta-1)y + \Dot 1^h &\text{if $x>0$}
	,\\
		-\bigl((\beta-1)y + \Dot 1^{d-1}\bigr) &\text{if $x<0$ and $h=d-1$}
	,\\
		-\bigl((\beta-1)y + \Dot 1^{d-1-h}\bigr) &\text{if $x<0$ and $1\leq h\leq d-2$}	
	.\end{cases}
\]
Since $\EB{\frac{1}{\beta-1}\times\Dot 1^j} = (1^j0^{d-j})^\omega$, the result follows from Lemma~\ref{lem:1}.

We finish by proving other direction.
Suppose $x>0$ satisfies \eqref{eq:x-per}.
Without the loss of generality, suppose that the length of the pre-period is a multiple of $d$,
 and put $y\eqdef (\beta-1)\times\Dot(x_{n+1}x_{n+2}\dotsm x_{n+d})^\omega
 = \Dot x_{n+1}x_{n+2}\dotsm x_{n+d} \in\class{h}$.
Then
\[
	\frac{x-y}{\beta-1} = \Dot (x_1{-}x_{n+1})\dotsm (x_d{-}x_{n+d})
	(x_{d+1}{-}x_{n+1})\dotsm (x_{n}{-}x_{n+d})0^\omega\in\Z[\beta]
.\]
Therefore $x\in\class{y}=\class{h}$.
The result for $x<0$ follows from the fact that $-\class h=\class{-h}=\class{d-1-h}$.
\end{proof}

\begin{proof}[of Theorem~\ref{thm:P}]
Let $x\in\Z[\beta]\cap\XS$.
By Lemmas~\ref{lem:B-S} and~\ref{lem:h}, the symmetric expansion $\ES{\abs x}$ is periodic with period $d$.
Suppose it is purely periodic.
Then by Lemma~\ref{lem:B-S}, $\EB{\frac{\abs{x}}{\beta-1}}$ is also purely periodic;
 we denote it $\EB{\frac{\abs{x}}{\beta-1}}=(p_1p_2\dotsm p_d)^\omega$.
Therefore, since $\frac{1}{\beta-1}=1\Dot(0^{d-1}1)^\omega$, we have that $\abs{x}=\Dot p_1p_2\dotsm p_d$.
The fact that $p_1=0$ follows from $\abs{x}\leq\frac12<\frac1\beta$.

On the other hand, any $x=\pm\Dot 0p_2\dotsm p_d\neq0$ satisfies that $x\in\XS\cap\Z[\beta]$
 and $\EB{\frac{\abs x}{\beta-1}}=(0p_2\dotsm p_d)^\omega$ is purely periodic, therefore $x\in\PP$.
\end{proof}

\begin{lemma}\label{lem:leq}
Let $\beta$ be a $d$-Bonacci number.
There exists a number $z\in\Z[\beta]$ such that $\Phi(z)$ lies exactly in $d-1$ tiles.
\end{lemma}

Before we prove this lemma, let us recall a helpful result by C.~Kalle and W.~Steiner:

\begin{lemma}\textup{\cite[Proposition~4.15]{KS}}\label{lem:KS}
Suppose $z\in\Z[\beta]\cap [0,\infty)$.
Let $k\in\N$ be an integer such that for all $y\in\PP$, the expansions $\ES{y}$
 and $\ES{y+\beta^{-k}z}$ have a common prefix at least as long as the period of $y$.

Then $\Phi(z)$ lies in a tile $\RR(x)$ for $x\in\Z[\beta]\cap\XS$ if and only if
\[
	x=\TS^k(y+\beta^{-k}z)
	\quad\text{for some $y\in\PP$}
.\]
\end{lemma}

\begin{proof}[of Lemma~\ref{lem:leq}]
We put $z \eqdef (0^{d-1}1)^{d-1}\Dot\in\Z[\beta]\cap[0,\infty)$.
Let us fix $y=\pm\,\Dot 0y_2y_3\dotsm y_d\allowbreak\in\PP$.
Then we can write $y$ as $y = (-p_1)\Dot p_1p_2p_3\dotsm p_d$, where
\[
	p_i = \begin{cases}
		y_i & \text{if $y>0$}
	,\\
		1-y_i & \text{if $y<0$}
	\end{cases}
\]
 (we put $y_1\eqdef0$).
Note that $h\eqdef p_1+p_2+\dots+p_d \in\{1,\dots,d-1\}$.
Let
\[
	t \eqdef \psi\bigl(y+\beta^{-{d^2}}z\bigr)
	= \frac{1}{\beta-1} \times \Dot p_1p_2\dotsm p_d \underbrace{(0^{d-1}1)(0^{d-1}1)\dotsm (0^{d-1}1)}_{\text{$d-1$ times}}
.\]
Defining $f(x) \eqdef \beta^d x + \frac{1}{\beta-1}$, we get that
\[
	\beta^{d^2} t
	= f^{d-1}\Bigl( \frac{1}{\beta-1}\times p_1p_2\dotsm p_d\Dot \Bigr)
	= f^{d-1}\bigl( p_1\dotsm p_d\Dot(p_1\dotsm p_d)^\omega \bigr)
,\]
 where $f^{d-1}(x)$ denotes the $(d-1)$th iteration $f(f(\cdots f(x)\cdots))$.

We have that $1\Dot(0^{d-1}1)^\omega=\frac{1}{\beta-1}=\sum_{i\geq1}\beta^{-i}=\Dot 1^\omega$
 and $1=\Dot(1^{d-1}0)^\omega$,
 from which we derive the following relations:
For $1\leq n\leq d-1$ and $x_1\dotsm x_{n-1}\neq 1^{n-1}$ we have
\begin{multline}\label{eq:f-1}
	f\bigl(x_{-N}\dotsm x_0\Dot(x_1\dotsm x_{n-1} 01^{d-n})^\omega\bigr)
\\\begin{aligned}
	= x_{-N}\dotsm x_0 x_1\dotsm x_{n-1} 01^{d-n}&\Dot(x_1\dotsm x_{n-1} 01^{d-n})^\omega + 0\Dot 1^\omega
\\
	= x_{-N}\dotsm x_0 x_1\dotsm x_{n-1} 0^{d-n+1}&\Dot(x_1\dotsm x_{n-1} 01^{d-n})^\omega + 1^{d-n}\Dot 1^\omega
\\
	= x_{-N}\dotsm x_0 x_1\dotsm x_{n-1} 0^{d-n+1}&\Dot(x_1\dotsm x_{n-1} 01^{d-n})^\omega + 10^{d-n}\Dot (0^{n-1}10^{d-n})^\omega
\end{aligned}\\
	= x_{-N}\dotsm x_0 x_1\dotsm x_{n-1} 10^{d-n}\Dot(x_1\dotsm x_{n-1} 1^{d-n+1})^\omega
.\end{multline}
We also have
\begin{multline}\label{eq:f-2}
	f\bigl(x_{-N}\dotsm x_0\Dot(1^{n-1}01^{d-n})^\omega\bigr)
\\\begin{aligned}
	= x_{-N}\dotsm x_0 1^{n-1}01^{d-n} &\Dot (1^{n-1}01^{d-n})^\omega + 1\Dot (0^{d-1}1)^\omega
\\
	= x_{-N}\dotsm x_0 1^{n-1}10^{d-n} &\Dot 0^\omega + 1\Dot (0^{d-1}1)^\omega
\end{aligned}\\
	= x_{-N}\dotsm x_0 1^n0^{d-n-1}1 \Dot (0^{d-1}1)^\omega
\end{multline}
 and
\begin{multline}\label{eq:f-3}
	f\bigl(1^{d-1}0\Dot (1^{d-1}0)^\omega\bigr)
%\\
	= 1^{d-1}01^{d-1}0\Dot (1^{d-1}0)^\omega + 1\Dot (0^{d-1}1)^\omega
\\
	= 1^{d-1}01^d\Dot 1^\omega
	= 1^d 0^{d-1}1 \Dot (0^{d-1}1)^\omega
.\end{multline}
After each iteration of $f$, the digit sum of the period either grows by one (in~\eqref{eq:f-1}) or goes from $d-1$ to $1$ (in~\eqref{eq:f-2} and~\eqref{eq:f-3})
 and the new period becomes $(0^{d-1}1)^\omega$.
It follows that
\begin{align*}
	f^{d-h}\bigl(p_1\dotsm p_d\Dot(p_1\dotsm p_d)^\omega\bigr)
	&= (\text{something})\Dot(0^{d-1}1)^\omega
,\\
	\beta^{d^2} t = f^{d-1}\bigl(p_1\dotsm p_d\Dot(p_1\dotsm p_d)^\omega\bigr)
	&= t_1t_2\dotsm t_{d^2}\Dot(0^{d-h}1^h)^\omega
.\end{align*}
Since the right-hand sides of~\eqref{eq:f-1}--\eqref{eq:f-3} contain neither $0^{d+1}$ nor $1^{d+1}$ as a factor,
 this sequence is $\TB$-admissible,
 therefore $\EB{t} = \EB{\psi(y+\beta^{-d^2}z)} = t_1t_2\dotsm t_{d^2}(0^{d-h}1^h)^\omega$.

By Lemma~\ref{lem:KS}, $\Phi(z)$ lies in the tile $\RR(x)$ for
\[
	x = \TS^{d^2}(y+\beta^{-d^2}z)
	= \psi^{-1} \TB^{d^2}(t)
.\]
Since $\EB{\TB^{d^2}t}=(0^{d-h}1^h)^\omega$, Lemma~\ref{lem:B-S} gives that $\ES{x}=(0^{d-h-1}10^{h-1}\1)^\omega$.
 
Finally, considering all $y\in\PP$ at once, we conclude that $\Phi(z)$
 lies exactly in tiles $\RR(\Dot(0^{d-h-1}10^{h-1}\1)^\omega)$ for $h\in\{1,2,\dots,d-1\}$.
That makes $d-1$ tiles.
\end{proof}

\begin{figure}
\includegraphics[width=\linewidth]{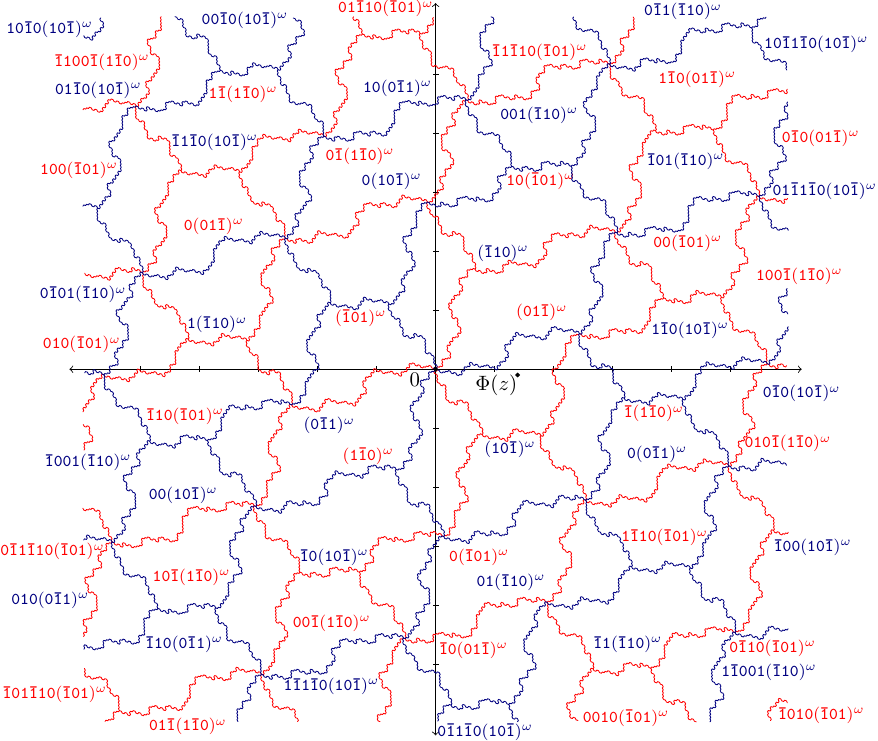}
\caption[The double tiling for the case $d=3$.]{The double tiling for the case $d=3$.
The layer $\LL_1$ is depicted in red and $\LL_2$ in blue.
We see that $\Phi(z)=1+\Phi(\beta^3)\in\RR(\Dot(10\1)^\omega)\cap\RR(\Dot(01\1)^\omega)$.}
\label{fig:mul3}
\end{figure}

\begin{example}
For $d=3$, there are $6$ purely periodic points $y\in\PP$.
Following the construction of $t$ in the previous proof we get the following
 (values of $x$ are the tiles in which $\Phi(z)=1+\Phi(\beta^3)$ lies):
\[
\begin{tabular}{
	>{$\Dot}l<{$}
	>{$\Dot}r<{$} @{} >{$(}l<{)^\omega$}
	>{\kern\tabcolsep$\Dot}l<{$}
	@{}>{${}=}l<{$}}
\multicolumn{1}{c}{$y$} & \multicolumn{2}{c}{$t$} & \multicolumn{2}{c}{$x$ such that $\Phi(z)\in\RR(x)$}  \\\midrule
001 &   001010101 & 001 & 001 & \Dot(01\1)^\omega \\
010 &   010011101 & 001 & 001 \\
011 &   011101010 & 011 & 011 & \Dot(10\1)^\omega \\
00\1 &  111001010 & 011 & 011 \\ 
0\10 &  110001010 & 011 & 011 \\
0\1\1 & 100110001 & 001 & 001 \\
\end{tabular}
\]
This is in accordance with the previous lemma and also with Figure~\ref{fig:mul3},
 where $\Phi(z)$ is shown and really lies in $\RR(\Dot(01\1)^\omega)$ and $\RR(\Dot(10\1)^\omega)$.

For $d=4$, we depict a cut through the multiple tiling in Figure~\ref{fig:mul4}.
\end{example}

\begin{lemma}\label{lem:geq}
Let $\beta$ be a $d$-Bonacci number.
For each point $z\in\Z[\beta]$ and for each $h\in\{1,2,\dots,d-1\}$
 there exists $x\in \LL_h$ such that $\Phi(z)\in\RR(x)$, where $\LL_h$ is given by~\eqref{eq:L}.
\end{lemma}

\begin{figure}
\includegraphics[width=\linewidth]{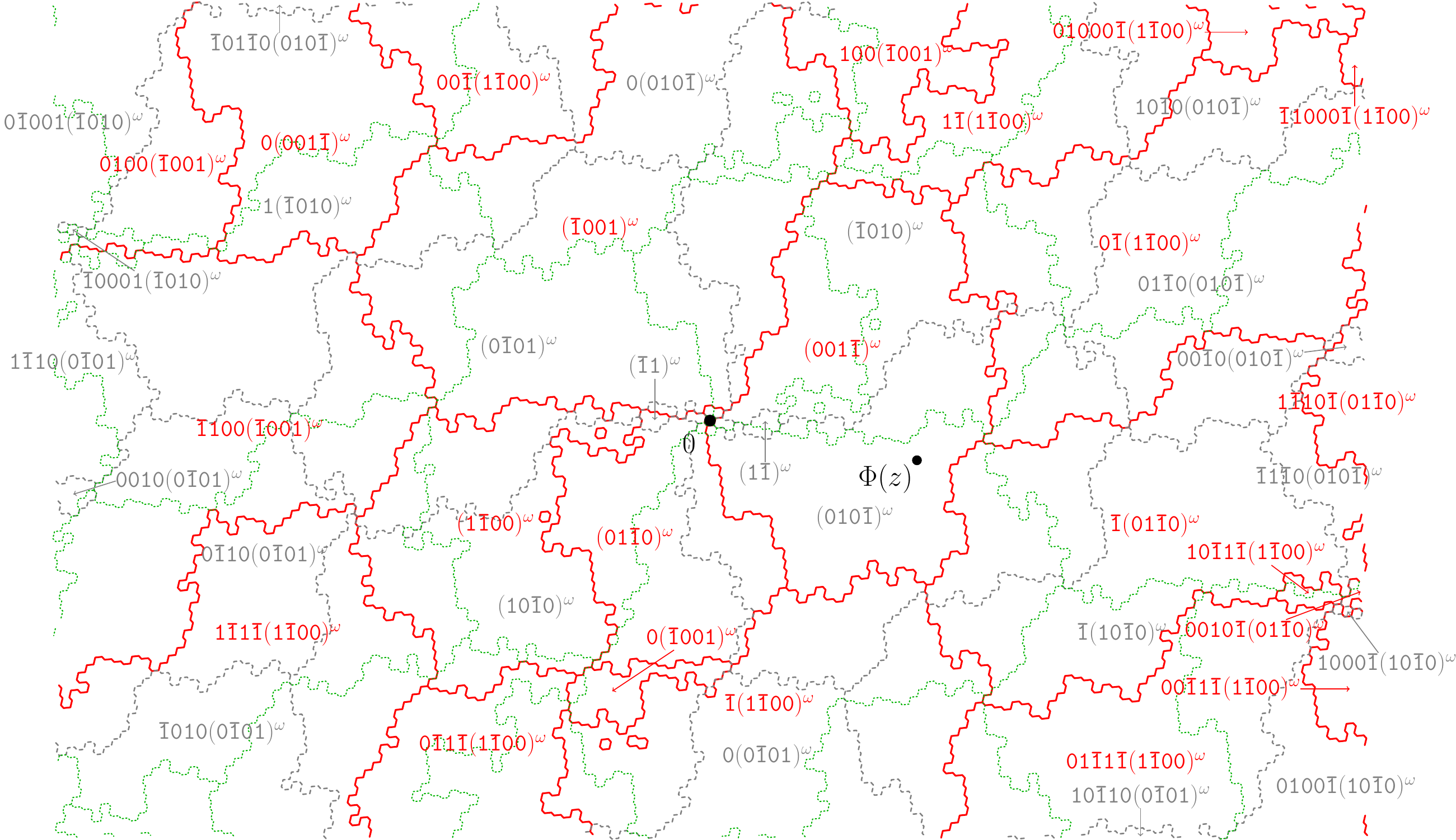}
\caption[A cut through the triple tiling for $d=4$.]{A cut through the triple tiling for $d=4$ that contains the point $\Phi(z)=1+\Phi(\beta^4)+\Phi(\beta^8)$.
Each layer is depicted in different style and colour: $\LL_1$ in solid red,
 $\LL_2$ in dashed gray, and $\LL_3$ in dotted green.
Since $\LL_3=-\LL_1$, the labels for $\LL_3$ are omitted.}
\label{fig:mul4}
\end{figure}

\begin{proof}
Suppose $z\geq0$.
Let $k\in\N$ satisfy the hypothesis of Lemma~\ref{lem:KS}.
Let $y\eqdef\Dot01^j\in\PP$, with $j\in\{1,\dots,d-1\}$ such that $y+\beta^{-k}z\in\class{h}$.
Denote $\ES{y+\beta^{-k}z}= x_1x_2\dotsm$
 and $\EB{\psi(y+\beta^{-k}z)}= t_1t_2\dotsm$.
Then $\Phi(z)$ lies in $\RR(x)$ for $x\eqdef \TS^k(y+\beta^{-k}z) = \Dot x_{k+1}x_{k+2}\dotsm$,
 and $x\in\class{h-\Dot x_0x_1\dotsm x_{k}}$.
From Lemma~\ref{lem:B-S} we have that $t_1=0$ and $t_{k+1}=1\Leftrightarrow x<0$.
Then
\begin{multline*}
	x \in \class{h-\Dot x_1\dotsm x_k}
	= \class{h-(\Dot t_2\dotsm t_kt_{k+1}-\Dot t_1t_2\dotsm t_k)}
\\[-1ex]
	= \class{h-(t_{k+1}-t_1)}
	= \class{h-t_{k+1}}
	= \begin{cases}
		\class{h} &\text{if $x>0$}
	,\\
		\class{h-1} &\text{if $x<0$}
	,\end{cases}
\end{multline*}
which means that $x\in\LL_h$.

If $z<0$, we already know that there exists $-x\in\LL_{d-h}$ such that $\Phi(-z)\in\RR(-x)$,
 hence $\Phi(z)\in\RR(x)$.
Since $-\class{h}=\class{d-1-h}$, we get that $\LL_{d-h}=-\LL_h$,
 therefore $x\in\LL_h$.
\end{proof}

\begin{proof}[of Theorem~\ref{thm:mt}]
The collection of tiles $\TT=\set{\RR(x)}{x\in\Z[\beta]\cap\XS}$ is a multiple tiling by Theorem~4.10 of~\cite{KS}.
By Lemma~\ref{lem:geq}, the degree is at least $d-1$ since all points of $\Phi(\Z[\beta])$ lie in at least that many tiles.
By Lemma~\ref{lem:leq}, the degree is at most $d-1$ since there exists a point that lies in only $d-1$ tiles.
\end{proof}

\begin{proof}[of Theorem~\ref{thm:L}]
By Lemma~\ref{lem:geq}, each $\Phi(z)$ for $z\in\Z[\beta]$ lies in at least one tile $\RR(x)$, $x\in\LL_h$,
 therefore --- since $\Phi(\Z[\beta])$ is dense in $\R^{d-1}$ and $\RR(x)$ is a closure of its interior --- 
 $\bigcup_{x\in\LL_h}\RR(x)=\R^{d-1}$.
Suppose there exists $M\subset\R^{d-1}$ of positive measure such that all $x\in M$ lie in at least two tiles of $\LL_h$.
These points lie in another $d-2$ tiles, one for each $\tilde h\in\{1,2,\dots,d-1\}\setminus\{h\}$.
Therefore the points of $M$ are covered by $d$ tiles, which is a contradiction with Theorem~\ref{thm:mt}.
\end{proof}

We finish by the proof of Theorem~\ref{thm:tiling}.
In this theorem, we need distinguish Rauzy fractals for $\TS$ as defined in \S\,\ref{sect:Rauzy}
 and Rauzy fractals for $\TB$ that are defined analogously.
To this end, we distinguish $\RS$, $\RB$, $\muS$ and $\muB$
 for the Rauzy fractals and invariant measures for $\TS$ and $\TB$, respectively.
Also, we note that in general, the support of the invariant measure $\muS$ for $\TS$ is a subset of $\XS$
 (not necessarily the whole $\XS$).
However, tiles for $x\in\XS\setminus\supp\muS$ have zero measure and excluding them allows
 us to use Theorem~4.10 of Kalle and Steiner"_\cite{KS} (see Remark~4.12 therein).
As $\psi$ is a~conjugacy (Lemma~\ref{lem:B-S}), we know that $\supp \muB=\psi(\supp\muS)$.

\begin{proof}[of Theorem~\ref{thm:tiling}]
Denote $\CS$ the tiling condition for $\TS$,
 $\CN$ the condition $N(\beta-1)=\pm1$ and $\CB$ the tiling condition for $\TB$.
We will show that $\CN\wedge\CB\Rightarrow\CS$, $\neg\CN\Rightarrow\neg\CS$, and $\neg\CB\Rightarrow\neg\CS$.

Before proceeding with the implications, we show that
\[
	\RS(\psi^{-1}y) = \Phi(\beta-1)\circ\RB(y)
	\quad\text{for all $y\in\Z[\beta]\cap\XB$}
,\]
 where $(\circ)$ is the component-wise product in $\C^e\times \R^{d-2e-1}$.
For a fixed $y$, let $C_n\eqdef\beta^n\TS^{-n}\psi^{-1}y$;
 then $\RS(\psi^{-1}y)=\lim_{n\to\infty} \Phi(C_n)$.
Defining $\theta:\XB\to\{0,1\}$ by $\psi^{-1}x = (\beta-1)x-\theta x$ we get that
\[
	C_n = \beta^n\psi^{-1}\TB^{-n}y
	= \set[\big]{(\beta-1)\beta^n z-\beta^n\theta z}{z\in \TB^{-n}y}
.\]
As $n\to\infty$, we have $\Phi(\beta^n)\to0$, therefore we may omit the term $\beta^n\theta z$ and write
\begin{multline*}
	\RS(\psi^{-1}y)
	= \lim_{n\to\infty} \Phi\bigl((\beta-1)\beta^n\TB^{-n}y\bigr)
\\[-1ex]
	= \Phi(\beta-1) \circ \lim_{n\to\infty} \Phi\bigl(\beta^n\TB^{-n}y\bigr)
	= \Phi(\beta-1) \circ \RB(y)
.\end{multline*}

Direction ($\CN\wedge\CB\Rightarrow\CS$).
As $\beta-1$ is a unit, we have that $\psi^{-1}(\Z[\beta]\cap\XB) = \Z[\beta]\cap\XS$,
 whence $\{\RS(x)\}_{x\in \Z[\beta]\cap\XS} = \{\Phi(\beta-1)\circ\RB(y)\}_{y\in\Z[\beta]\cap\XB}$.
As the map $\vec{v}\mapsto \Phi(\beta-1)\circ\vec{v}$ is a linear bijection $\R^{d-1}\to\R^{d-1}$
 and $\{\RB(y)\}_{y\in\Z[\beta]\cap\XB}$ is a tiling, we conclude
 that $\{\RS(x)\}_{x\in \Z[\beta]\cap\XS}$ is a tiling as well.

Direction ($\neg\CN\Rightarrow\neg\CS$).
By the same argument as above, we have that $\{\RS(x)\}_{x\in\mathcal K}$,
 where $\mathcal K\eqdef\psi^{-1}(\Z[\beta]\cap\XB)$, is a tiling or a multiple tiling, i.e., it covers $\R^{d-1}$.
As $\beta-1$ is not a unit and $\Z[\beta]$ is dense in $\R$, we have also that
 $(\beta-1)\Z[\beta]$ is dense hence $1+(\beta-1)\Z[\beta]\subseteq\Z[\beta]\setminus (\beta-1)\Z[\beta]$ is dense.
Therefore there exists $x\in\Z[\beta]\cap\supp\muS$ such that $x\notin\mathcal K$.
Then $\RS(x)$ is a set of positive measure that is covered at least twice:
 once by $\{\RS(x)\}_{x\in\mathcal K}$ and once by $\RS(x)$.
We conclude that $\{\RS(x)\}_{x\in \Z[\beta]\cap\XS}$ is not a tiling.

Direction ($\neg\CB\Rightarrow\neg\CS$).
By the same argument as above, we have that $\{\RS(x)\}_{x\in\mathcal K}$
 --- which is a subset of the multiple tiling $\{\RS(x)\}_{x\in \Z[\beta]\cap\XS}$ ---
 is a multiple tiling of $\R^{d-1}$ of covering degree $\geq2$, because $\{\RB(y)\}_{y\in\Z[\beta]\cap\XB}$ is.
\end{proof}

\section{Open Problems}\label{sect:problems}

\begin{problem'}
Take a $(d,a)$-Bonacci number for $d\geq2$ and $a\geq2$, i.e.,
 the Pisot number $\beta\in(a,a+1)$ satisfying $\beta^d = a\beta^{d-1} + \dots + a\beta + a$.
What is the number of layers of the multiple tiling for the symmetric $\beta$-transformation in this case?
\end{problem'}

\begin{problem'}
Consider the $d$-Bonacci number $\beta$, and the transformation $T_{\beta,l}\colon [l,l+1),\, x\mapsto \beta x-\qfl{\beta x-l}$.
We know that $T_{\beta,0}$ induces a tiling"_\cite{barge_2016a}.
We prove here that $T_{\beta,-1/2}$ induces a multiple tiling with covering degree $d-1$.
What happens if $-\frac12<l<0$?
What are the possible values of the covering degree?
\end{problem'}

\begin{problem'}
For $\varrho$ the Pisot root of $x^3-x-1$, we have that $\varrho-1$ is a unit,
 but we also have that $\TS$ induces a double tiling"_\cite[\S\,4.5.2]{KS}.
From Theorem~\ref{thm:tiling} we conclude that the $\TB$ does not induce a single tiling.
We ask the following: Is there any $\gamma\in(\varrho,2)$ such that $\TB$
 induces a single tiling for all Pisot units $\beta\in(\gamma,2)$?

Note that $\TB$ induces a single tiling for all $d$-Bonacci numbers
 as $\psi^{-1}(\Z[\beta]\cap\XB)=\LL_0$ and we know that $\{\RS(x)\}_{x\in\LL_0}$ is a tiling.
Furthermore, for the other two cubic Pisot units $\beta\in(1,2)$, namely roots of $x^3-2x^2+x-1$ and $x^3-x^2-1$,
 we know that $\TS$ induces a single tiling"_\cite[\S\,4.5.2]{KS} hence $\TB$ also induces a single tiling by Theorem~\ref{thm:tiling}.
\end{problem'}

\begin{problem'}
Tackle the tilings for the symmetric $\beta$-expansions for $\beta>2$.
\end{problem'}

\section*{Acknowledgements}
\begingroup\footnotesize
We acknowledge support by Czech Science Foundation (GA\v{C}R) grant 17-04703Y and ANR/FWF project
 ``FAN -- Fractals and Numeration'' (ANR-12-IS01-0002, FWF grant I1136).
We are also grateful for Sage and Ti\textit kZ software which were used to prepare the figures"_\cite{sage,tikz}.
\par\endgroup

\bibliographystyle{amsalpha}
\bibliography{biblio}

\end{document}